\documentclass[letterpaper, 10 pt, conference]{ieeeconf}  % Comment this line out
\usepackage{multirow}
\usepackage{graphicx}
\usepackage{amsmath}

\allowdisplaybreaks
\usepackage{amsthm}
\usepackage{mathtools}
\usepackage[psamsfonts]{amssymb}
\usepackage{amsxtra}
\usepackage{hyperref}
\usepackage{paralist}
\usepackage{threeparttable}
\usepackage{cleveref}
\usepackage{cite}
\usepackage{color}
\usepackage{url}
\usepackage[justification=centering]{caption}

\renewcommand{\thesection}{\arabic{section}}

\newtheorem{Theorem}{Theorem}[section]
\newtheorem{Lemma}[Theorem]{Lemma}
\newtheorem{proposition}[Theorem]{Proposition}

\theoremstyle{definition}
\newtheorem{assump}{Assumption}

\newenvironment{myassump}[2][]
{\renewcommand\theassump{#2}\begin{assump}[#1]}
	{\end{assump}}

\catcode`~=11 % make LaTeX treat tilde (~) like a normal character
\newcommand{\urltilde}{\kern -.15em\lower .7ex\hbox{~}\kern .04em}
\catcode`~=13 % revert back to treating tilde (~) as an active character
\IEEEoverridecommandlockouts                              % This command is only
\title{\LARGE \bf
Distributed Zeroth Order Optimization Over Random Networks: A Kiefer-Wolfowitz Stochastic Approximation Approach
}

\author{Anit~Kumar~Sahu, Dusan~Jakovetic, Dragana~Bajovic and Soummya~Kar% <-this % stops a space
\thanks{The work of DJ and DB is supported in part by the EU Horizon 2020 project I-BiDaaS, project number 780787. The work of DJ is also supported in part by the Serbian Ministry of Education, Science, and Technological Development, grant 174030. The work of AKS and SK was supported in part by National Science Foundation under grant CCF-1513936.}% <-this % stops a space
\thanks{D. Bajovic is with the Faculty of Technical Sciences, University of Novi Sad 21000 Novi Sad, Serbia
        {\tt\small dbajovic@uns.ac.rs}}%
\thanks{D. Jakovetic is with the Department of Mathematics and Informatics, Faculty of Sciences, University of Novi Sad 21000 Novi Sad, Serbia
	{\tt\small djakovet@uns.ac.rs}}
\thanks{A. K. Sahu and S. Kar are with the Department of Electrical and Computer Engineering, Carnegie Mellon University, Pittsburgh, PA 15213
        {\tt\small \{anits,soummyak\}@andrew.cmu.edu}}%
}

\begin{document}

\maketitle
\thispagestyle{empty}
\pagestyle{empty}

%%%%%%%%%%%%%%%%%%%%%%%%%%%%%%%%%%%%%%%%%%%%%%%%%%%%%%%%%%%%%%%%%%%%%%%%%%%%%%%%
%
%
\begin{abstract}
We study a standard distributed 
optimization framework where $N$ 
networked nodes collaboratively 
minimize the sum of their local 
convex costs. The main body 
of existing work considers the 
described problem when the underling 
network is either static or deterministically 
varying, and 
the distributed optimization algorithm 
is of first or second order, i.e., 
it involves the local costs' gradients 
and possibly the local Hessians. 
 In this paper, we consider 
  the currently understudied but highly relevant 
  scenarios when: 1) only noisy function values' 
  estimates are available (no gradients nor Hessians 
  can be evaluated); and 2) 
 the underlying network 
 is randomly varying 
 (according to an independent, 
 identically distributed process).   
  For the described random networks-zeroth order optimization setting, 
  we develop a distributed stochastic 
  approximation method of the Kiefer-Wolfowitz 
  type. Furthermore, 
  under standard smoothness and strong convexity assumptions 
  on the local costs, 
  we establish the $O(1/k^{1/2})$ 
  mean square convergence rate for the method -- 
  the rate that matches 
  that of the method's centralized counterpart under equivalent 
  conditions.
\end{abstract}

%%%%%%%%%%%%%%%%%%%%%%%%%%%%%%%%%%%%%%%%%%%%%%%%%%%%%%%%%%%%%%%%%%%%%%%%%%%%%%%%

\section{Introduction}
\label{section-intro}
\noindent We consider a commonly studied distributed optimization setting where $N$ nodes are interconnected in a generic, connected network, and they collectively minimize the sum of their local convex costs with respect to a global (vector-valued) variable of common interest. There has been a significant and increasing interest in the described distributed optimization problems, 
e.g.,~\cite{DistributedMirrorDescent,RabbatDistributedStronglyCVX,Kozat,NedicStochasticPush}, which include algorithms which have access to a stochastic first order or a second order oracle. In that, every query made to the oracle gives an unbiased estimate of the gradient or the Hessian based on whether its a first order or second order oracle. Moreover, in the context of distributed optimization, most existing work in the literature consider static or deterministically varying networks. \\
\noindent In this paper, we consider the  currently largely understudied, but highly relevant case of 1) zeroth order optimization (only noisy function values, and no gradients nor Hessians are available); and 2) randomly varying networks, more precisely, the networks modeled through a sequence of independent identically distributed (i.i.d.) symmetric Laplacian matrices, 
such that the network is connected on average. Regarding the former,  zeroth order methods become highly desirable when the functions of interest are not given in analytical forms or evaluating the gradient or the Hessian is expensive. For the latter, random network models are more adequate than deterministically varying or static models when the networked nodes communicate through unreliable wireless links, like, e.g., with many emerging internet of things applications, including, for example, maintenance and monitoring of industrial manufacturing systems or large scale industrial plants where communication environments may be harsh.\\
\noindent Our main contributions are as follows. We propose a distributed stochastic approximation method of the Kiefer-Wolfowitz type~(see, for example \cite{kiefer1952stochastic}). The method utilizes standard strategy in distributed (sub)gradient-like methods, where the iterations consist of 1)local estimates' weight averaging across nodes' neighborhoods (consensus); and 2) 
a negative step with respect to the Kiefer-Wolfowitz-type estimates of local functions gradients (innovations). We show that, by a careful design of the consensus and innovations time-varying 
weights, the distributed Kiefer-Wolfowitz method achieves the $O(1/k^{1/2})$ mean square convergence rate. This rate is achieved for twice continuously differentiable, convex local costs with bounded Hessians, assuming only the availability of noisy function values' estimates, with zero-mean and finite-second moment noises. The achieved $O(1/k^{1/2})$ rate of the distributed method 
is highest possible and matches that of the counterpart centralized Kiefer-Wolfowitz stochastic approximation method and the minimax rate for the aforementioned class of cost functions~(see, \cite{duchi2015optimal}).\\
We now briefly review the literature to help us contrast this paper from prior work. Zeroth order optimization, where the stochastic oracle can be queried for only noisy function values has been applied to scenarios involving black-box based optimization and high-dimensional optimization~(see, for example \cite{chen2017zoo,wang2017stochastic}). Various approaches to zeroth order optimization have been adopted such as the classical Kiefer Wolfowitz stochastic approximation~\cite{kiefer1952stochastic,kushner2003stochastic}, simultaneous perturbation stochastic approximation~\cite{spall1992multivariate} and other random direction random smoothing 
based variants such as \cite{nesterov2011random,duchi2015optimal,wang2017stochastic}. However, all the aforementioned references solve the minimization problem in a centralized framework or the setting where a single agent has all the data available to it. In the context of random networks, in~\cite{CDCSGD2018}, 
we consider a distributed stochastic gradient method and establish the method's $O(1/k^{2})$ convergence rate. Reference~\cite{CDCSGD2018} complements the current paper by assuming 
that nodes have access to gradient estimates. Recently, in \cite{hajinezhad2017zeroth}, a distributed zeroth optimization algorithm was proposed for non-convex minimization with a static graph where a random directions random smoothing approach was employed. In contrast, in this paper we solve a distributed zeroth optimization algorithm for a strongly convex minimization employing a Kiefer-Wolfowitz type stochastic approximation with a random sequence of graphs.\\
\noindent\textbf{Paper organization}.\\The next
paragraph introduces notation.
Section \ref{sec:model} describes the model and the stochastic gradient method we consider.
Section~\ref{sec:perf_analysis} states and proves the main result on the algorithm's MSE convergence rate.
Section~\ref{sec:sim} provides a simulation example. Finally, we conclude in Section~\ref{sec:conc}.

\noindent\textbf{Notation}.\\
We denote by $\mathbb R$ the set of real numbers and by ${\mathbb R}^m$ the $m$-dimensional
Euclidean real coordinate space. We use normal lower-case letters for scalars,
lower case boldface letters for vectors, and upper case boldface letters for
matrices. Further, we denote by: $\mathbf{A}_{ij}$ the entry in the $i$-th row and $j$-th column of
a matrix $\mathbf{A}$;
$\mathbf{A}^\top$ the transpose of a matrix $A$; $\otimes$ the Kronecker product of matrices;
$I$, $0$, and $\mathbf{1}$, respectively, the identity matrix, the zero matrix, and the column vector with unit entries; $\mathbf{J}$ the $N \times N$ matrix $J:=(1/N)\mathbf{1}\mathbf{1}^\top$.
When necessary, we indicate the matrix or vector dimension as a subscript.
Next, $A \succ  0 \,(A \succeq  0 )$ means that
the symmetric matrix $A$ is positive definite (respectively, positive semi-definite).
We denote by:
$\|\cdot\|=\|\cdot\|_2$ the Euclidean (respectively, spectral) norm of its vector (respectively, matrix) argument; $\lambda_i(\cdot)$ the $i$-th smallest eigenvalue; $\nabla h(w)$ and $\nabla^2 h(w)$ the gradient and Hessian, respectively, evaluated at $w$ of a function $h: {\mathbb R}^m \rightarrow {\mathbb R}$, $m \geq 1$; $\mathbb P(\mathcal A)$ and $\mathbb E[u]$ the probability of
an event $\mathcal A$ and expectation of a random variable $u$, respectively. $\mathbf{e}_{j}$ denotes the $j$-th column on the identity matrix $\mathbf{I}$ where the dimension is made clear from the context.
Finally, for two positive sequences $\eta_n$ and $\chi_n$, we have: $\eta_n = O(\chi_n)$ if
$\limsup_{n \rightarrow \infty}\frac{\eta_n}{\chi_n}<\infty$.      

\section{Model, Algorithm, and Preliminaries}
\label{sec:model}
The network of $N$ agents in our setup collaboratively aim to solve the following unconstrained problem:
{\small\begin{align}
\label{eq:opt_problem}
\min_{\mathbf{x}\in\mathbb{R}^{d}}\sum_{i=1}^{N}f_{i}(\mathbf{x}),
\end{align}}
where $f_{i}: \mathbb{R}^{d}\mapsto\mathbb{R}$ is a convex
function available to node $i$, $i=1,...,N$. We make the following assumption on the functions $f_{i}(\cdot)$:

\begin{myassump}{A1}
	\label{as:1}
	For all $i=1,...,N$, function
	$f_{i}: \mathbb{R}^{d}\mapsto\mathbb{R}$ is twice continuously differentiable with Lipschitz continuous gradients. In particular, there exist constants $L, \mu > 0$ such that for all $\mathbf{x} \in {\mathbb R}^d$, 
	{\small\begin{align*}
	\mu\ \mathbf{I} \preceq \nabla^2 f_i(\mathbf{x}) \preceq L\mathbf{I}.
	\end{align*}}
\end{myassump}

From Assumption~\ref{as:1} we have that each $f_i$, $i=1,\cdots,N$, is strongly convex with modulus $\mu$. Using standard properties of convex functions, we have for any $\mathbf{x,y} \in {\mathbb R}^d$:
{\small\begin{align*}
&f_i(\mathbf{y}) \geq f_i(\mathbf{x})+ \nabla f_i(\mathbf{x})^\top \,(\mathbf{y}-\mathbf{x})+ \frac{\mu}{2}\|\mathbf{x}-\mathbf{y}\|^2,\\
&\|\nabla f_i(\mathbf{x}) -\nabla f_i(\mathbf{y})\| \leq L\,\|\mathbf{x}-\mathbf{y}\|.
\end{align*}}

\noindent We also have that from assumption \ref{as:1}, the optimization problem in \eqref{eq:opt_problem} has a unique solution, which we denote by $\mathbf{x}^{\ast}\in\mathbb{R}^{d}$. Throughout the paper, we use the sum function which is defined as $f:\,{\mathbb R}^d \rightarrow \mathbb R$,~$f(\mathbf{x})=\sum_{i=1}^N f_i(\mathbf{x})$.

\noindent We consider distributed stochastic zeroth order optimization to solve~\eqref{eq:opt_problem} over random networks. Inter-agent communication is modeled by a sequence of independent and identically distributed~(i.i.d.)~undirected random networks: at each time instant $k=0,1,...$,
the underlying inter-agent communication network is denoted by $\mathcal{G}(k) = (V,\mathbf{E}(k))$, with $V=\{1,...,N\}$ being the set of nodes and $\mathbf{E}(k)$ being the random set of undirected edges. The edge connecting node $i$ and $j$ is denoted as $\{i,j\}$. The time-varying random neighborhood of node $i$ at time $k$ (excluding node~$i$) is represented as $\Omega_i(k)=\{j \in V:\,\,\{i,j\} \in \mathbf{E}(k)\}$. The graph Laplacian of the random graph $\mathcal{G}(k)$ at time $k$ is given by $\mathbf{L}(k)\in\mathbb{R}^{N\times N}$, where $\mathbf{L}(k)$ is given by $\mathbf{L}_{ij}(k) =-1,$ if $\{i,j\} \in \mathbf{E}(k)$, $i \neq j$; $\mathbf{L}_{ij}(k) =0,$ if $\{i,j\} \notin \mathbf{E}(k)$, $i \neq j$; and~$\mathbf{L}_{ii}(k) =-\sum_{j \neq i}\mathbf{L}_{ij}(k).$ It is to be noted that the Laplacian at each time instant is symmetric and a positive semidefinite matrix. As the considered graph sequence is i.i.d., we have that $\mathbb{E}[\, \mathbf{L}(k)\,]=\overline{\mathbf{L}}$. Let the graph corresponding to $\overline{\mathbf{L}}$ be given by $\overline{\mathcal{G}} = (V,\overline{\mathbf{E}})$.

\noindent We make the following assumption on $\overline{\mathcal{G}}$.
\begin{myassump}{A2}
	\label{as:2}
	\label{assumption-network}
	The inter-agent communication graph is connected on average, i.e., $\overline{\mathcal{G}}$ is connected. In other words, $\lambda_{2}(\mathbf{\overline{\mathcal{L}}}) > 0$.
\end{myassump}
\noindent We denote by $|\mathcal{L}|$ the cardinality of a set of Laplacians chosen from the total number of possible Laplacians~(necessarily finite) so as to ensure $\underline{p} = \inf_{\mathbf{L}\in\mathcal{L}} \mathbb{P}\left(\mathbf{L}(t)=\mathbf{L}\right) > 0$.

\subsection{Distributed Kiefer Wolfowitz type Optimization}
\label{sec:kwsa}
We employ a distributed Kiefer Wolfowitz stochastic approximation~(KWSA) type method to solve~\eqref{eq:opt_problem}. Each node~$i$, $i=1,...,N$, in our setup maintains a local copy of its local estimate of the optimizer $\mathbf{x}_i(k) \in {\mathbb R}^d$ at all times. In order to carry out the optimization, each agent $i$ makes queries to a stochastic zeroth order oracle at time $k$, from which the agent obtains noisy function values of $f_i(\mathbf{x}_i(k))$. Denote the noisy value of $f_{i}(\cdot)$ as $\widehat{f}_{i}(\cdot)$ where,
\begin{align}
\label{eq:func_dist}
\widehat{f}_{i}(\mathbf{x}_{i}(k))=f_{i}(\mathbf{x}_{i}(k))+\widehat{v}_{i}(k).
\end{align}
Due to the unavailability of the analytic form of the functionals, the gradient can not be evaluated and hence, we resort to a gradient approximation. In order to approximate the gradient, each agent makes two calls to the stochastic zeroth order oracle corresponding to each dimension. For instance, for dimension $j\in\{1,\cdots,d\}$ agent $i$ queries for $f_i(\mathbf{x}_i(k)+c_{k}\mathbf{e}_{j})$ and $f_i(\mathbf{x}_i(k)-c_{k}\mathbf{e}_{j})$ at time $k$ and obtains $\widehat{f}_i(\mathbf{x}_i(k)+c_{k}\mathbf{e}_{j})$ and $\widehat{f}_i(\mathbf{x}_i(k)-c_{k}\mathbf{e}_{i})$ respectively, where $c_{k}$ is a carefully chosen time-decaying potential~(to be specified soon). Denote by $\mathbf{g}_{i}(\mathbf{x}_{i}(k))$ the approximated gradient, obtained as for each $j\in\{1,\cdots,d\}$ :
\begin{align}
\label{eq:KW_grad}
&\mathbf{e}_{j}^{\top}\mathbf{g}_{i}(\mathbf{x}_{i}(k)) = \frac{\widehat{f}_{i}\left(\mathbf{x}_{i}(k)+c_{k}\mathbf{e}_{j}\right)-\widehat{f}_{i}\left(\mathbf{x}_{i}(k)-c_{k}\mathbf{e}_{j}\right)}{2c_{k}}\nonumber\\
&\Rightarrow\mathbf{e}_{j}^{\top}\mathbf{g}_{i}(\mathbf{x}_{i}(k)) = \frac{f_{i}\left(\mathbf{x}_{i}(k)+c_{k}\mathbf{e}_{j}\right)}{2c_{k}}\nonumber\\&-\frac{f_{i}\left(\mathbf{x}_{i}(k)-c_{k}\mathbf{e}_{j}\right)}{2c_{k}}+\frac{\hat{v}^{+}_{i,j}(k)-\hat{v}^{-}_{i,j}(k)}{2c_{k}},
\end{align}
where $\hat{v}^{+}_{i,j}(k)$ and $\hat{v}^{-}_{i,j}(k)$ denote the measurement noise corresponding to the measurements $\widehat{f}_{i}\left(\mathbf{x}_{i}(k)+c_{k}\mathbf{e}_{j}\right)$ and $\widehat{f}_{i}\left(\mathbf{x}_{i}(k)-c_{k}\mathbf{e}_{j}\right)$ respectively. The vectors $\widehat{\mathbf{v}}^{+}_{i}(k)\in\mathbb{R}^{d}$ and $\widehat{\mathbf{v}}^{-}_{i}(k)\in\mathbb{R}^{d}$ stack all the component wise measurement noise at a node $i$ and are given by $\widehat{\mathbf{v}}^{+}_{i}(k)=\left[\hat{v}^{+}_{i,1}(k),\cdots,\hat{v}^{+}_{i,N}(k)\right]$ and $\widehat{\mathbf{v}}^{-}_{i}(k)=\left[\hat{v}^{-}_{i,1}(k),\cdots,\hat{v}^{-}_{i,N}(k)\right]$ respectively. For the rest of the paper, we define $\mathbf{v}_{i}(k) \doteq \left(\widehat{\mathbf{v}}^{+}_{i}(k)-\widehat{\mathbf{v}}^{-}_{i}(k)\right)/2$.
\noindent Using the mean value theorem, we have,
\begin{align}
\label{eq:KW_grad_1}
\mathbf{g}_{i}(\mathbf{x}_{i}(k)) = \nabla f(\mathbf{x}_{i}(k))+c_{k} \mathbf{P}_{i}(\mathbf{x}_{i}(k))+\frac{\mathbf{v}_{i}(k)}{c_{k}},
\end{align}
where
{\small\begin{align*}
&\mathbf{e}_{j}^{\top}\mathbf{P}_{i}(\mathbf{x}_{i}(k)) = \frac{\mathbf{e}_{j}^{\top}\nabla^{2}f(\mathbf{x}_{i}(k)+c_{k}\alpha_{i,j}^{+}\mathbf{e}_{j})\mathbf{e}_{j}}{2}\nonumber\\&-\frac{\mathbf{e}_{j}^{\top}\nabla^{2}f(\mathbf{x}_{i}(k)-c_{k}\alpha_{i,j}^{-}\mathbf{e}_{j})\mathbf{e}_{j}}{2},
\end{align*}}
\noindent where $0\le \alpha_{i,j}^{+}, \alpha_{i,j}^{-}\le 1$. Finally, for arbitrary deterministic initializations $\mathbf{x}_i(0) \in {\mathbb R}^d$, $i=1,...,N$, the optimizer update rule at node $i$ and $k=0,1,...,$ is given as follows:
\begin{align}
\label{eq:update_rule_node}
&\mathbf{x}_i(k+1)=\mathbf{x}_i(k) - \beta_k\sum_{j \in \Omega_i(k)} \left( \mathbf{x}_i(k) - \mathbf{x}_j(k) \right)\nonumber\\
&-\alpha_k\mathbf{g}_{i}(\mathbf{x}_{i}(k)).
\end{align}
It is to be noted that unlike first order stochastic gradient methods, where the algorithm has access to unbiased estimates of the gradient. The local gradient estimates $\mathbf{g}_{i}(\cdot)$ used in \eqref{eq:update_rule_node} are biased (see \eqref{eq:KW_grad_1}) due to the unavailability of the exact gradient functions and their approximations using the zeroth order scheme in \eqref{eq:KW_grad}. The update is carried on in all agents parallely in a synchronous fashion. The weight sequences $\{\alpha_{k}\}$, $\{c_{k}\}$ and $\{\beta_{k}\}$ are given by $\alpha_{k}=\alpha_0/(k+1)$, $c_{k}=c_0/(k+1)^{\delta}$ and $\beta_{k}=\beta_0/(k+1)^{\tau}$ respectively, where $\alpha_0, c_0, \beta_0 > 0$. We state an assumption on the weight sequences before proceeding further.
\begin{myassump}{A3}
\label{as:3}
The constants $\alpha_0$, $\delta > 0$ and $\tau\in (0,1)$ are chosen such that,
\begin{align}
\label{eq:a3}
\sum_{k=1}^{\infty} \frac{\alpha_{k}^2}{c_k^2} < \infty, \mu\alpha_{0} < 1.
\end{align}	
\end{myassump}
\noindent Denote by $\mathbf{x}(k) = \left[\mathbf{x}_{1}^{\top}(k),\cdots,\mathbf{x}_{N}^{\top}(k)\right]^{\top}\in\mathbb{R}^{Nd}$, $\mathbf{P}(\mathbf{x}(k)) = \left[\mathbf{P}_{1}^{\top}\left(\mathbf{x}_{1}(k)\right),\cdots,\mathbf{P}_{N}^{\top}\left(\mathbf{x}_{N}(k)\right)\right]^{\top}\in\mathbb{R}^{Nd}$ 
 the vectors that stacks the local optimizers and the gradient bias terms~(see  \eqref{eq:KW_grad_1})of all nodes.
 Also, define function $F: \mathbb{R}^{Nd}\mapsto\mathbb{R}$, by
 $F(\mathbf{x})=\sum_{i=1}^{N}f_{i}(\mathbf{x}_{i})$,
 with $\mathbf{x} = \left[\mathbf{x}_{1}^{\top},\cdots,\mathbf{x}_{N}^{\top}\right]^{\top}\in\mathbb{R}^{Nd}.$
 Finally, let $\mathbf{W}_{k}=\left(\mathbf{I}-\mathbf{L}_{k}\right)\otimes \mathbf{I}_{d}$, where $\mathbf{L}_{k}=\beta_k\,\mathbf{L}(k)$.
Then the update in \eqref{eq:update_rule_node} can be written as:
\begin{align}
\label{eq:update_rule}
&\mathbf{x}(k+1) = \mathbf{W}_{k}\mathbf{x}(k)\nonumber\\&-\alpha_{k}\left(\nabla F(\mathbf{x}(k))+c_{k}\mathbf{P}(\mathbf{x}(k))+\frac{\mathbf{v}(k)}{c_{k}}\right).
\end{align}
\noindent Let $\mathcal{F}_k$ denote the history of the proposed algorithm up to time $k$. Given that the sources of randomness in our algorithm are the noise sequence $\{\mathbf{v}(k)\}$ and the random network sequence $\{\mathbf{L}_{k}\}$, $\mathcal{F}_k$ is given by the $\sigma$-algebra generated by the collection of random variables $\{\,\mathbf{L}(s),\,\mathbf{v}_i(s)\}$,~$i=1,...,N$,~$s=0,...,k-1$.
\begin{myassump}{A4}
  	\label{as:4}
  	For each $i=1,...,N$,
  	the sequence of measurement noises $\{\widehat{\mathbf{v}}_i(k)\}$
  	satisfies for all $k=0,1,...$:
  	\begin{align}
  	\label{eq:as4}
  	&\mathbb{E}[\,\widehat{\mathbf{v}}_i(k)\,|\,\mathcal{F}_k\,] =0,\,\,\mathrm{almost\,surely\,(a.s.)}\nonumber\\
  	&\mathbb{E}[\,\|\widehat{\mathbf{v}}_i(k)\|^2\,|\,\mathcal{F}_k\,] \leq c_{f}\|\mathbf{x}_i(k)\|^2+\sigma^{2},
  	\,\,\mathrm{a.s.},
  	\end{align}
  	where $c_f$ and $\sigma^{2}$ are nonnegative constants.
  	\end{myassump}
\noindent It is to be noted that assumption~\ref{as:4} is trivially satisfied, when $\{\mathbf{v}_i(k)\}$ is an i.i.d. zero-mean, finite second moment, noise sequence such that $\mathbf{v}_i(k)$ is also independent of the history~$\mathcal{F}_k$. However, the assumption allows the noise to be dependent on the current iterate at all times.
\section{Performance Analysis}
\label{sec:perf_analysis}
\subsection{Main Result and Auxiliary Lemmas}
We state the main result concerning the mean square error at each agent $i$ next.
\begin{Theorem}
	\label{theorem-1}
	1) Consider the optimizer estimate sequence $\{\mathbf{x}(k)\}$ generated by the algorithm \eqref{eq:update_rule_node}. Let assumptions \ref{as:1}-\ref{as:4} hold. Then, for each node~$i$'s optimizer estimate $\mathbf{x}_i(k)$ and the solution $\mathbf{x}^\star$ of problem~\eqref{eq:opt_problem}, $\forall k \geq k_3$ there holds:
	\begin{align}
	\label{eq:th1}
	&\mathbb{E}\left[\left\|\mathbf{x}_{i}(k)-\mathbf{x}^{\ast}\right\|^{2}\right] \le 2M_{k}+\frac{64N\Delta_{1,\infty}\alpha_{0}^{2}}{\mu c_0^2p_{\mathcal{L}}^{2}\beta_0^{2}(k+1)^{2-2\tau-2\delta}}\nonumber\\&\frac{4(L-\mu)^{2}N^{2}dc_{0}^{2}}{\mu(k+1)^{2\delta}}+2Q_{k}+
	\frac{8\Delta_{1,\infty}\alpha_{0}^{2}}{p_{\mathcal{L}}^{2}\beta_0^{2}c_0^2(k+1)^{2-2\tau-2\delta}}\nonumber\\&+
	\frac{4N\alpha_{0}\left(c_{f}q_{\infty}(N,d,\alpha_0,c_0)+N\sigma_1^{2}\right)}{c_0^2\mu(k+1)^{1-2\delta}},
	\end{align}
	where, $p_{\mathcal{L}}=\underline{p}\lambda_{2}\left(\overline{\mathbf{L}}\right)/|\mathcal{L}|$, $k_3 = \inf\left\{k~:~\beta_{k}^{2}<\frac{1}{N^{2}}, p_{\mathcal{L}}^{2}\beta_{k}^{2}<1\right\}$, $\Delta_{1,\infty}=6c_{f}q_{\infty}(N,d,\alpha_0,c_0)+6N\sigma_{1}^{2}$  and $q_{\infty}(N,d,\alpha_0,c_0)=\mathbb{E}\left[\left\|\mathbf{x}(k_0)-\mathbf{x}^{o}\right\|^{2}\right]+\frac{\sqrt{Nd}(L-\mu)\alpha_0c_0}{\delta}+\frac{Nd(L-\mu)^{2}\alpha_0^2c_0^2}{1+2\delta}+\frac{\alpha_0^2\left(2c_{f}N\left\|\mathbf{x}^{o}\right\|^{2}+N\sigma^{2}\right)}{c_0^{2}(1-2\delta)}+4\frac{\|\nabla F(\mathbf{x}^{o})\|^2}{\mu^2}$. In the latter $k_0$ is given by $k_{0}=\inf\left\{k~:~\frac{\mu}{2} > (L-\mu)\sqrt{Nd}c_{k}+\frac{2c_{f}\alpha_{k}}{c_{k}^{2}}\right\}.$~$M_{k}$ and $Q_{k}$ are terms which decay faster than the rest of the terms.\\
	2) In particular, the rate of decay of the RHS of \eqref{eq:th1} is given by $(k+1)^{-\delta_1}$, where $\delta_{1}=\min\left\{1-2\delta,2-2\tau-2\delta,2\delta\right\}$. By, optimizing over $\tau$ and $\delta$, we obtain that for $\tau=1/2$ and $\delta=1/4$ and hence,
	\begin{align*}
	&\mathbb{E}\left[\left\|\mathbf{x}_{i}(k)-\mathbf{x}^{\ast}\right\|^{2}\right]
	\le 2M_{k}+\frac{64N\Delta_{1,\infty}\alpha_{0}^{2}}{\mu c_0^2p_{\mathcal{L}}^{2}\beta_0^{2}(k+1)^{0.5}}\nonumber\\&\frac{4(L-\mu)^{2}N^{2}dc_{0}^{2}}{\mu(k+1)^{0.5}}+2Q_{k}+
	\frac{8\Delta_{1,\infty}\alpha_{0}^{2}}{p_{\mathcal{L}}^{2}\beta_0^{2}c_0^2(k+1)^{0.5}}\nonumber\\&+
	\frac{4N\alpha_{0}\left(c_{f}q_{\infty}(N,d,\alpha_0,c_0)+N\sigma_1^{2}\right)}{c_0^2\mu(k+1)^{0.5}}= O\left(\frac{1}{k^{\frac{1}{2}}}\right),~~\forall i.
	\end{align*}
\end{Theorem}
\noindent Theorem~\ref{theorem-1} establishes the $O(1/k^{1/2})$ MSE rate of convergence of the algorithm~\eqref{eq:update_rule_node}; due to the assumed $f_i$'s strong convexity, the theorem also implies that $\mathbb{E}\left[ f(\mathbf{x}_i(k)) - f(\mathbf{x}^\star)\right] =O(1/k^{1/2})$. Note that the expectation in Theorem~\ref{theorem-1} is both with respect to randomness in gradient noises and with respect to the randomness in the underlying network. The $O(1/k^{1/2})$ rate is independent of the statistics of the underlying random network, as long as the network is connected on average.\\
\noindent From \eqref{eq:th1}, it might seem that the dependence of the upper bound is linear in terms of $d$. However, on tuning the constants $\alpha_0\asymp d^{-1/5}$, $\beta_0\asymp d^{-1/10}$ and $c_0\asymp d^{-3/10}$, the dependence of $\mathbb{E}\left[\left\|\mathbf{x}_{i}(k)-\mathbf{x}^{\ast}\right\|^{2}\right]$ can be reduced to $d^{2/5}$. It is to be noted that the upper bound derived in \eqref{eq:th1} matches with that of the minimax bound for (centralized)~zeroth order optimization with twice continuously differentiable cost functions as derived in \cite{duchi2015optimal}. The sublinear rate of convergence of zeroth order optimization algorithms in the context of KWSA can be attributed to the biased gradients. For better finite time convergence rates, bias-reduction techniques such as the ``twicing trick'' and finite difference interpolation techniques can be used.\\
\noindent\textbf{Proof strategy and auxiliary lemmas}.
Establishing the main result in Theorem~\ref{theorem-1} involves three crucial steps which are outlined in the subsections \ref{sec:bound}, \ref{sec:dis_bound} and \ref{sec:opt_gap}.
Subsection \ref{sec:bound} concerns with the mean square boundedness of the iterates $\mathbf{x}_i(k)$, which also implies the mean square boundedness of the gradients $\nabla f_i(\mathbf{x}_i(k))$. In subsection \ref{sec:dis_bound}, the mean square error of the disagreements of a node's optimizer estimate with respect to the network averaged optimizer estimate ,i.e., $\overline{\mathbf{x}}(k):=\frac{1}{N}\sum_{i=1}^N \mathbf{x}_i(k)$, is characterized in terms of $k$ and the algorithm parameters. Finally, subsection \ref{sec:opt_gap} characterizes the optimality gap of the networked average optimizer estimate sequence with respect to the optimizer of \eqref{eq:opt_problem} and on combining the result from subsection \ref{sec:dis_bound}, the result follows.

\subsection{Mean square boundedness of the iterate sequence}
\label{sec:bound}
This subsection shows the mean square boundedness of the algorithm iterates.
\begin{Lemma}
\label{Lemma-MSS-BDD}
Let the hypotheses of Theorem \ref{theorem-1} hold. In addition assume that, $\left\|\nabla F(\mathbf{1}_{N}\otimes\mathbf{x}^{\ast})\right\|$ is bounded.
Then, we have,
\begin{align*}
&\mathbb{E}\left[\left\|\mathbf{x}(k)-\mathbf{x}^{o}\right\|^2\right]\le q_{k_0}(N,d,\alpha_0,c_0)\nonumber\\&+\frac{\sqrt{Nd}(L-\mu)\alpha_0c_0}{\delta}+\frac{Nd(L-\mu)^{2}\alpha_0^2c_0^2}{1+2\delta}\nonumber\\&+\frac{\alpha_0^2\left(2c_{f}N\left\|\mathbf{x}^{o}\right\|^{2}+N\sigma^{2}\right)}{c_0^{2}(1-2\delta)}+4\frac{\|\nabla F(\mathbf{x}^{o})\|^2}{\mu^2}\nonumber\\&\doteq q_{\infty}(N,d,\alpha_0,c_0),
\end{align*}
where $\mathbb{E}\left[\left\|\mathbf{x}(k_0)-\mathbf{x}^{o}\right\|^{2}\right] \le q_{k_0}(N,d,\alpha_0,c_0)$ and $k_{0}=\inf\left\{k~:~\frac{\mu}{2} > (L-\mu)\sqrt{Nd}c_{k}+\frac{2c_{f}\alpha_{k}}{c_{k}^{2}}\right\}.$
\end{Lemma}
\begin{proof}
\begin{align}
\label{eq:update_rule1}
&\mathbf{x}(k+1) = \mathbf{W}_{k}\mathbf{x}(k)\nonumber\\&-\frac{\alpha_{k}}{c_{k}}\left(c_{k}\nabla F(\mathbf{x}(k))+c_{k}^{2}P(\mathbf{x}(k))+\mathbf{v}(k)\right).
\end{align}
Denote $\mathbf{x}^{o} = \mathbf{1}_{N}\otimes x^{\ast}$.
Then, we have,
\begin{align}
\label{eq:update_rule2}
&\mathbf{x}(k+1)-\mathbf{x}^{o} = \mathbf{W}_{k}(\mathbf{x}(k)-\mathbf{x}^{o})\nonumber\\&-\alpha_{k}\left(\nabla F(\mathbf{x}(k))-\nabla F(\mathbf{x}^{o})\right)\nonumber\\&-\frac{\alpha_{k}}{c_{k}}\mathbf{v}(k)-\alpha_{k}\nabla F(\mathbf{x}^{o})-\alpha_{k}c_{k}\mathbf{P}(\mathbf{x}(k)).
\end{align}
By Leibnitz rule, we have,
\begin{align}
\label{eq:mvt}
&\nabla F(\mathbf{x}(k))-\nabla F(\mathbf{x}^{o}) \nonumber\\&= \left[\int_{s=0}^{1}\nabla^{2}F\left(\mathbf{x}^{o}+s(\mathbf{x}(k)-\mathbf{x}^{o})\right)ds\right]\left(\mathbf{x}(k)-\mathbf{x}^{o}\right)\nonumber\\
&=\mathbf{H}_{k}\left(\mathbf{x}(k)-\mathbf{x}^{o}\right).
\end{align}
By Lipschitz continuity of the gradients and strong convexity of $f(\cdot)$, we have that $L\mathbf{I}\succcurlyeq\mathbf{H}_{k}\succcurlyeq\mu\mathbf{I}$. 
Denote by $\boldsymbol{\zeta}(k) = \mathbf{x}(k)-\mathbf{x}^{o}$
and by $\boldsymbol{\xi}(k) = \left(\mathbf{W}_{k}-\alpha_{k}\mathbf{H}_{k}\right)(\mathbf{x}(k)-\mathbf{x}^{o})
-\alpha_k \,\nabla F(\mathbf{x}^{o})$.
Then, there holds:
{\small\begin{align}
&\mathbb{E}[\,\|\boldsymbol{\zeta}(k+1)\|^2 \,|\,\mathcal{F}_k\,]
\leq
\|\boldsymbol{\xi}(k)\|^2 \nonumber \\
&-2 \alpha_k \,{\boldsymbol{\xi}}(k)^\top
\mathbb{E}[\,\mathbf{v}(k) \,|\,\mathcal{F}_k\,] +
\alpha_k^2 \,\mathbb{E}[\,\|\mathbf{v}(k)\|^2 \,|\,\mathcal{F}_k\,] \nonumber \\
&+\alpha_{k}^{2}c_{k}^{2}\mathbf{P}^{\top}(\mathbf{x}(k))\mathbf{P}(\mathbf{x}(k))-2\alpha_{k}c_{k}\mathbf{P}^{\top}(\mathbf{x}(k)){\boldsymbol{\xi}}(k)\nonumber\\
&+\mathbf{P}\left(\mathbf{x}(k)\right)^\top \mathbb{E}\left[\mathbf{v}(k)|\mathcal{F}_k\right].
\label{eqn-combine-1}
\end{align}}
We use the following inequalities:
{\small\begin{align}
\label{eq:cross_term}
&-2\alpha_{k}c_{k}\mathbf{P}^{\top}(\mathbf{x}(k))\left(\mathbf{W}_{k}-\alpha_{k}\mathbf{H}_{k}\right)(\mathbf{x}(k)-\mathbf{x}^{o})\nonumber\\&\le 2\alpha_{k}c_{k}\left\|\mathbf{P}(\mathbf{x}(k))\right\|\left\|\mathbf{W}_{k}-\alpha_{k}\mathbf{H}_{k}\right\|\left\|\mathbf{x}(k)-\mathbf{x}^{o}\right\|\nonumber\\
&\le \sqrt{Nd}(L-\mu)\alpha_{k}c_{k}\left(1-\mu\alpha_{k}\right)\left(1+\left\|\mathbf{x}(k)-\mathbf{x}^{o}\right\|^{2}\right)\nonumber\\
&\le \sqrt{Nd}(L-\mu)\alpha_{k}c_{k}+\sqrt{Nd}(L-\mu)\alpha_{k}c_{k}\left\|\mathbf{x}(k)-\mathbf{x}^{o}\right\|^{2},
\end{align}}
{\small\begin{align}
\label{eq:bias}
\alpha_{k}^{2}c_{k}^{2}\mathbf{P}^{\top}(\mathbf{x}(k))\mathbf{P}(\mathbf{x}(k)) \le Nd\left(L-\mu\right)^{2}\alpha_{k}^{2}c_{k}^{2},
\end{align}}
and 
{\small\begin{align}
\label{eq:variance}
&\frac{\alpha_{k}^{2}}{c_{k}^{2}}\mathbb{E}\left[\left\|\mathbf{v}(k)\right\|^{2}|\mathcal{F}_{k}\right]\le \frac{\alpha_{k}^{2}}{c_{k}^{2}}c_{f}N\left\|\mathbf{x}(k)\right\|^{2}+\frac{\alpha_{k}^{2}}{c_{k}^{2}}N\sigma^{2}\nonumber\\
&\le 2\frac{\alpha_{k}^{2}}{c_{k}^{2}}c_{f}\left\|\mathbf{x}(k)-\mathbf{x}^{o}\right\|^{2}+\frac{\alpha_{k}^{2}}{c_{k}^{2}}\left(2c_{f}\left\|\mathbf{x}^{o}\right\|^{2}+N\sigma^{2}\right).
\end{align}}
Then from \eqref{eqn-combine-1}, we have,
\begin{align}
\label{eq:eqn-combine-2}
&\mathbb{E}[\,\|\boldsymbol{\zeta}(k+1)\|^2 \,|\,\mathcal{F}_k\,]\le \|\boldsymbol{\xi}(k)\|^2\nonumber\\
&+\sqrt{Nd}(L-\mu)\alpha_{k}c_{k}\|\boldsymbol{\zeta}(k)\|^2+ 2\frac{\alpha_{k}^{2}}{c_{k}^{2}}c_{f}\|\boldsymbol{\zeta}(k)\|^2\nonumber\\
&+\frac{\alpha_{k}^{2}}{c_{k}^{2}}\left(2c_{f}\left\|\mathbf{x}^{o}\right\|^{2}+N\sigma^{2}\right)+\sqrt{Nd}(L-\mu)\alpha_{k}c_{k}\nonumber\\
&+Nd\left(L-\mu\right)^{2}\alpha_{k}^{2}c_{k}^{2}
\end{align}

We next bound $\|\boldsymbol{\xi}(k)\|^2$.
Note that
$\|\mathbf{W}_k -\alpha_k \,\boldsymbol{H}_k\| \leq 1-\mu\,\alpha_k$.
Therefore, we have:
\begin{equation}
\label{eqn-xi-zeta}
\|\boldsymbol{\xi}(k)\| \leq (1-\mu\,\alpha_k)\,\|\boldsymbol{\zeta}(k)\|
+ \alpha_k\, \|\nabla F(\mathbf{x}^{o})\|.
\end{equation}
We now
use the following inequality:
\begin{align}
\label{eq:cool_ineq_1}
(a+b)^{2} \leq \left(1+\theta\right)a^{2}+\left(1+\frac{1}{\theta}\right)b^{2},
\end{align}
for any $a,b \in \mathbb{R}$ and $\theta > 0$.
We set $\theta=\mu\alpha_{k}$.
Using the inequality \eqref{eq:cool_ineq_1} in \eqref{eqn-xi-zeta}, we have:
{\small\begin{align}
	\label{eq:update_rule4}
	&\left\|\boldsymbol{\xi}(k)\right\|^{2}
	\le \left(1+\mu\alpha_{k}\right)(1-\alpha_k\mu)^{2}\left\|\boldsymbol{\zeta}(k)\right\|^2\nonumber\\
	&+\left(1+\frac{1}{\mu\alpha_{k}}\right)\alpha_{k}^{2}\|\nabla F(\mathbf{x}^{o})\|^2\nonumber\\
	&\le (1-\alpha_k\mu)\left\|\boldsymbol{\zeta}(k)\right\|^2+2\frac{\alpha_k}{\mu}\|\nabla F(\mathbf{x}^{o})\|^2.
\end{align}}

Using \eqref{eq:update_rule4} in \eqref{eq:eqn-combine-2}, we have,
{\small\begin{align}
\label{eq:combine-3}
&\mathbb{E}[\,\|\boldsymbol{\zeta}(k+1)\|^2 \,|\,\mathcal{F}_k\,]\le+2\frac{\alpha_k}{\mu}\|\nabla F(\mathbf{x}^{o})\|^2+ \left\|\boldsymbol{\zeta}(k)\right\|^2\nonumber\\&\times\left(1-\alpha_k\mu+\sqrt{Nd}(L-\mu)\alpha_{k}c_{k}+2\frac{\alpha_{k}^{2}}{c_{k}^{2}}c_{f}\right)\nonumber\\
&+\frac{\alpha_{k}^{2}}{c_{k}^{2}}\left(2c_{f}\left\|\mathbf{x}^{o}\right\|^{2}+N\sigma^{2}\right)+\sqrt{Nd}(L-\mu)\alpha_{k}c_{k}\nonumber\\
&+Nd\left(L-\mu\right)^{2}\alpha_{k}^{2}c_{k}^{2}.
\end{align}}

\noindent Define $k_0$ as follows:
\begin{align*}
k_{0}=\inf\left\{k~:~\frac{\mu}{2} > (L-\mu)\sqrt{Nd}c_{k}+\frac{2c_{f}\alpha_{k}}{c_{k}^{2}}\right\}.
\end{align*}
It is to be noted that $k_0$ is necessarily finite as $c_{k}\to 0$ and $\alpha_{k}c_{k}^{-2}\to 0$ as $k\to\infty$. \begin{proposition}
	\label{prop:res1}
	Let the hypotheses of Theorem \ref{theorem-1} hold. Then, we have $\forall k \geq k_0$,
	\begin{align*}
	&\mathbb{E}\left[\|\boldsymbol{\zeta}(k+1)\|^2 \right] \le q_{k_0}(N,d,\alpha_0,c_0)\nonumber\\&+\frac{\sqrt{Nd}(L-\mu)\alpha_0c_0}{\delta}+\frac{Nd(L-\mu)^{2}\alpha_0^2c_0^2}{1+2\delta}\nonumber\\&+\frac{\alpha_0^2\left(2c_{f}N\left\|\mathbf{x}^{o}\right\|^{2}+N\sigma^{2}\right)}{c_0^{2}(1-2\delta)}+4\frac{\|\nabla F(\mathbf{x}^{o})\|^2}{\mu^2}\nonumber\\&\doteq q_{\infty}(N,d,\alpha_0,c_0)
	\end{align*}
\end{proposition}
\begin{proof}
	{\small\begin{align}
		\label{eq:combine-4}
		&\mathbb{E}\left[\|\boldsymbol{\zeta}(k+1)\|^2 \right] \le \prod_{l=k_0}^{k}\left(1-\frac{\mu\alpha_{l}}{2}\right)\mathbb{E}\left[\|\boldsymbol{\zeta}(k_0)\|^2\right]\nonumber\\
		&+4\frac{\|\nabla F(\mathbf{x}^{o})\|^2}{\mu^2}\sum_{l=k_0}^{k}\left(\prod_{m=l+1}^{k}\left(1-\frac{\mu\alpha_{m}}{2}\right)-\prod_{m=l}^{k}\left(1-\frac{\mu\alpha_{m}}{2}\right)\right)
		\nonumber\\&+\sum_{l=k_0}^{k}\left(\frac{\alpha_{l}^{2}}{c_{l}^{2}}\left(2c_{f}\left\|\mathbf{x}^{o}\right\|^{2}+N\sigma^{2}\right)\right.\nonumber\\&\left.+Nd\left(L-\mu\right)^{2}\alpha_{l}^{2}c_{l}^{2}+\sqrt{Nd}(L-\mu)\alpha_{l}c_{l}\right)\nonumber\\
		&\Rightarrow \mathbb{E}\left[\|\boldsymbol{\zeta}(k+1)\|^2 \right] \le q_{k_0}(N,d,\alpha_0,c_0)\nonumber\\&+\frac{\sqrt{Nd}(L-\mu)\alpha_0c_0}{\delta}+\frac{Nd(L-\mu)^{2}\alpha_0^2c_0^2}{1+2\delta}\nonumber\\&+\frac{\alpha_0^2\left(2c_{f}N\left\|\mathbf{x}^{o}\right\|^{2}+N\sigma^{2}\right)}{c_0^{2}(1-2\delta)}+4\frac{\|\nabla F(\mathbf{x}^{o})\|^2}{\mu^2}\nonumber\\&\doteq q_{\infty}(N,d,\alpha_0,c_0),
		\end{align}}
\end{proof}
\noindent From proposition \ref{prop:res1}, we have that $\mathbb{E}\left[\left\|\mathbf{x}(k+1)-\mathbf{x}^{o}\right\|^{2}\right]$ is finite and bounded from above, where $\mathbb{E}\left[\left\|\mathbf{x}(k_0)-\mathbf{x}^{o}\right\|^{2}\right] \le q_{k_0}(N,d,\alpha_0,c_0)$. From the boundedness of $\mathbb{E}\left[\left\|\mathbf{x}(k)-\mathbf{x}^{o}\right\|^2\right]$, we have also established the boundedness of $\mathbb{E}\left[\left\|\nabla F(\mathbf{x}(k))\right\|^2\right]$ and $\mathbb{E}\left[\left\|\mathbf{x}(k)\right\|^2\right]$. 

\end{proof}

\noindent With the above development in place, we can bound the variance of the noise process $\{\mathbf{v}(k)\}$ as follows:
\begin{align}
\label{eq:noise_variance_condition}
&\mathbb{E}\left[\left\|\mathbf{v}(k)\right\|^{2}|\mathcal{F}_{k}\right]\le 0.5\mathbb{E}\left[\left\|\widehat{\mathbf{v}}^{+}(k)\right\|^{2}|\mathcal{F}_{k}\right] \nonumber\\&+0.5\mathbb{E}\left[\left\|\widehat{\mathbf{v}}^{-}(k)\right\|^{2}|\mathcal{F}_{k}\right]\nonumber\\&\le 2c_{f}q_{\infty}(N,d,\alpha_0,c_0)+2N\underbrace{\left(\sigma^{2}+\left\|\mathbf{x}^{\ast}\right\|^{2}\right)}_{\text{$\sigma_{1}^{2}$}}.
\end{align}
%for any $a,b \in \mathbb{R}$ and $\theta > 0$. We set $\theta=\frac{c_0}{k+1}$. Using the inequality \eqref{eq:cool_ineq} in \eqref{eq:update_rule3}, we have,
%\begin{align}
%\label{eq:update_rule4}
%\left\|\mathbf{x}(k+1)-\mathbf{x}^{o}\right\|^{2} \le \left(1+\frac{c_0}{k+1}\right)(1-\alpha_0\mu)^{2}\left\|\mathbf{x}(k)-\mathbf{x}^{o}\right\|^2+\left(1+\frac{k+1}{c_0}\right)\alpha_{k}^{2}u^{2}(k).
%\end{align}
%Note that, $\mathbb{E}\left[u^{2}(k)\right]<\infty$. Then, for $c_0 < \alpha_0\mu$, we have,
%\begin{align}
%\label{eq:update_rule5}
%&\left\|\mathbf{x}(k+1)-\mathbf{x}^{o}\right\|^{2} \le \left(1-\frac{c_1}{k+1}\right)\left\|\mathbf{x}(k)-\mathbf{x}^{o}\right\|^2+\frac{c_2}{k+1}u^{2}(k)\nonumber\\
%&\Rightarrow \mathbb{E}\left[\left\|\mathbf{x}(k+1)-\mathbf{x}^{o}\right\|^{2}\right] \le \left(1-\frac{c_1}{k+1}\right)\mathbb{E}\left[\left\|\mathbf{x}(k)-\mathbf{x}^{o}\right\|^2\right]+\frac{c_3}{k+1}.
%\end{align}

\subsection{Disagreement Bounds}
\label{sec:dis_bound}
\noindent We now study the disagreement of the optimizer sequence $\{\mathbf{x}_{i}(k)\}$ at a node $i$ with respect to the~(hypothetically available) network averaged optimizer sequence, i.e., $\overline{\mathbf{x}}(k)=\frac{1}{N}\sum_{i=1}^{N}\mathbf{x}_{i}(k)$. Define the disagreement at the $i$-th node as $\widetilde{\mathbf{x}}_{i}(k)=\mathbf{x}_{i}(k)-\overline{\mathbf{x}}(k)$. The vectorized version of the disagreements $\widetilde{\mathbf{x}}_{i}(k),~i=1,\cdots,N$, can then be written as $\widetilde{\mathbf{x}}(k)=\left(\mathbf{I}-\mathbf{J}\right)\mathbf{x}(k)$, where $\mathbf{J}=\frac{1}{N}\left(\mathbf{1}_{N}\otimes\mathbf{I}_{d}\right)\left(\mathbf{1}_{N}\otimes\mathbf{I}_{d}\right)^{\top}=\frac{1}{N}\mathbf{1}_{N}\mathbf{1}_{N}^{\top}\otimes\mathbf{I}_{d}$.
We have the following Lemma:
\begin{Lemma}
	\label{lemma-disag-bound}
	Let the hypotheses of Theorem \ref{theorem-1} hold. Then, we have $\forall k \geq k_3$
	\begin{align*}
	&\mathbb{E}\left[\left\|\widetilde{\mathbf{x}}(k+1)\right\|^{2}\right] \le Q_{k}+ \frac{4\Delta_{1,\infty}\alpha_{0}^{2}}{p_{\mathcal{L}}^{2}\beta_0^{2}c_0^2(k+1)^{2-2\tau-2\delta}}\nonumber\\&=O\left(\frac{1}{k^{2-2\delta-2\tau}}\right),
	\end{align*}
	where $Q_k$ is a term which decays faster than $(k+1)^{-2+2\tau+2\delta}$, $k_{3} = \inf\left\{k~:~\beta_{k}^{2}<\frac{1}{N^{2}}, s^{2}(k) < 1\right\}$ and $p_{\mathcal{L}}=\underline{p}\lambda_{2}\left(\overline{\mathbf{L}}\right)/|\mathcal{L}|$.
\end{Lemma}
\noindent As detailed in the next Subsection, Lemma~\ref{lemma-disag-bound} plays a crucial role in providing a tight bound for the bias in the gradient estimates
according to which the global average $\overline{\mathbf{x}}(k)$ evolves.
\begin{proof}
The process $\{\widetilde{\mathbf{x}}(k)\}$ follows the recursion:
\begin{align}
\label{eq:dis1}
&\widetilde{\mathbf{x}}(k+1)=\widetilde{\mathbf{W}}_{k}\widetilde{\mathbf{x}}(k)\nonumber\\&-\frac{\alpha_{k}}{c_{k}}\left(\mathbf{I}-\mathbf{J}\right)\underbrace{\left(c_{k}\nabla F(\mathbf{x}(k))+\mathbf{v}(k)+c_{k}^{2}P\left(\mathbf{x}(k)\right)\right)}_{\text{$\mathbf{w}(k)$}},
\end{align}
where $\widetilde{\mathbf{W}}_{k} = \mathbf{W}_{k}-\mathbf{J}=\left(\mathbf{I}-\mathbf{L}_{k}\right)\otimes\mathbf{I}_{d}-\mathbf{J}$. 
%We also have that from Lemma {\color{red} Fill up Lemma number}, $\mathbb{E}\left[\left\|\mathbf{w}(k)\right\|^{2}|\right]\leq c_{7} < \infty$, which follows due to the boundedness of $\mathbb{E}\left[\left\|\nabla F(\mathbf{x}(k))\right\|^2\right]$. 
Then, we have,
\begin{align}
\label{eq:dis2}
\left\|\widetilde{\mathbf{x}}(k+1)\right\| \le \left\|\widetilde{\mathbf{W}}_{k}\right\| \left\|\widetilde{\mathbf{x}}(k)\right\|+\frac{\alpha_{k}}{c_k}\left\|\mathbf{w}(k)\right\|.
\end{align}
We now invoke Lemma 4.4 from \cite{kar2013distributed}, which lets us conclude that under assumption \ref{as:2}~$\forall k \geq k_{1}$,
\begin{align}
\label{eq:dis4}
\left\|\widetilde{\mathbf{x}}(k+1)\right\| \le (1-r(k))\left\|\widetilde{\mathbf{x}}(k)\right\|+\frac{\alpha_{k}}{c_{k}}\left\|\mathbf{w}(k)\right\|,
\end{align}
where $r(k)\in[0,1]$ a.s. is a $\mathcal{F}_{k}$ adapted process and satisfies
\begin{align}
\label{eq:rk}
\mathbb{E}\left[r(k)|\mathcal{F}_{k}\right] \geq \beta_{k}\underline{p}\frac{\lambda_{2}\left(\overline{\mathbf{L}}\right)}{|\mathcal{L}|}\doteq s(k)~a.s.
\end{align} 
Let $p_{\mathcal{L}}=\underline{p}\frac{\lambda_{2}\left(\overline{\mathbf{L}}\right)}{|\mathcal{L}|}$. Moreover, $k_{1}$ is given by
{\small\begin{align}
\label{eq:k1}
k_1 = \inf\left\{k~:~\beta_{k}^{2}<\frac{1}{N^{2}}\right\}.
\end{align}}
%\begin{align}
%\label{eq:le.5}
%\left\|\widetilde{\mathbf{W}}_{k}\right\| = 1- r(k),
%\end{align}
%where $r(k)\in[0,1]$ almost surely, and
Using \eqref{eq:cool_ineq_1} in \eqref{eq:dis4}, we have $\forall k \geq k_1$,
\begin{align}
\label{eq:dis5}
&\left\|\widetilde{\mathbf{x}}(k+1)\right\|^{2}\le \left(1+\theta_k\right)(1-r(k))^{2}\left\|\widetilde{\mathbf{x}}(k)\right\|^{2}\nonumber\\&+\left(1+\frac{1}{\theta_k}\right)\frac{\alpha_k^2}{c_{k}^{2}}\left\|\widetilde{\mathbf{w}}(k)\right\|^2,
\end{align}
for $\theta_k = s(k)$.
Then, we have,
\begin{align}
\label{eq:dis6}
&\mathbb{E}\left[\left\|\widetilde{\mathbf{x}}(k+1)\right\|^{2}|\mathcal{F}_{k}\right]\leq \left(1+\theta_k\right)(1-s(k))^{2}\left\|\widetilde{\mathbf{x}}(k)\right\|^{2}\nonumber\\&+\left(1+\frac{1}{\theta_k}\right)\frac{\alpha_k^2}{c_k^{2}}\mathbb{E}\left[\left\|\mathbf{w}(k)\right\|^{2}|\mathcal{F}_{k}\right],
\end{align}
where 
\begin{align}
\label{eq:wk_bound}
&\mathbb{E}\left[\left\|\mathbf{w}(k)\right\|^{2}|\mathcal{F}_{k}\right] \le 3c_{k}^{2}\left\|\nabla F(\mathbf{x}(k))\right\|^{2}+3\mathbb{E}\left[\left\|\mathbf{v}(k)\right\|^{2}|\mathcal{F}_{k}\right]\nonumber\\&+3c_{k}^{2}\left\|P\left(\mathbf{x}(k)\right)\right\|^{2}\nonumber\\
&\le 3c_{k}^{2}\left\|\nabla F(\mathbf{x}(k))\right\|^{2}+3c_{k}^{2}Nd(L-\mu)^{2}\nonumber\\&+6c_{f}q_{\infty}(N,d,\alpha_0,c_0)+6N\sigma_{1}^{2}\nonumber\\
&\Rightarrow \mathbb{E}\left[\left\|\mathbf{w}(k)\right\|^{2}\right] \le 3\left(2c_{f}+c_{k}^{2}L^{2}\right)q_{\infty}(N,d,\alpha_0,c_0)\nonumber\\&+3c_{k}^{2}Nd(L-\mu)^{2} +6N\sigma_{1}^{2}\nonumber\\
&= \underbrace{6c_{f}q_{\infty}(N,d,\alpha_0,c_0)+6N\sigma_1^{2}}_{\text{$\Delta_{1,\infty}$}}\nonumber\\&+\underbrace{3c_{k}^{2}Nd(L-\mu)^{2}+3c_{k}^{2}L^{2}q_{\infty}(N,d,\alpha_0,c_0)}_{\text{$c_{k}^{2}\Delta_{2,\infty}$}}\doteq\Delta_{k}\nonumber\\&\Rightarrow \mathbb{E}\left[\left\|\mathbf{w}(k)\right\|^{2}\right] < \infty.
\end{align}
With the above development in place, we then have,
\begin{align}
\label{eq:dis6.5}
&\mathbb{E}\left[\left\|\widetilde{\mathbf{x}}(k+1)\right\|^{2}\right]\leq \left(1+\theta_k\right)(1-s(k))^{2}\mathbb{E}\left[\left\|\widetilde{\mathbf{x}}(k)\right\|^{2}\right]\nonumber\\&+\left(1+\frac{1}{\theta_k}\right)\frac{\alpha_k^2}{c_{k}^{2}}\Delta_{k}.
\end{align}
\begin{proposition}
	\label{prop:res2}
	Let the hypotheses of Theorem \ref{theorem-1} hold. Then, we have $k\geq k_3=\max\{k_1,k_2\}$ where $k_2 = \inf\left\{k~:~s^{2}(k) < 1\right\}$,
	\begin{align*}
	\mathbb{E}\left[\left\|\widetilde{\mathbf{x}}(k+1)\right\|^{2}\right] \le Q_{k}+ \frac{4\Delta_{1,\infty}\alpha_{0}^{2}}{p_{\mathcal{L}}^{2}\beta_0^{2}c_0^2(k+1)^{2-2\tau-2\delta}}.
	\end{align*}
\end{proposition}

\begin{proof}
	From \eqref{eq:dis6.5}, we have for $k\geq k_3$,
	\begin{align}
	\label{eq:dis6.6}
	&\mathbb{E}\left[\left\|\widetilde{\mathbf{x}}(k+1)\right\|^{2}\right]\leq (1-s(k))\mathbb{E}\left[\left\|\widetilde{\mathbf{x}}(k)\right\|^{2}\right]\nonumber\\&+\left(1+\frac{1}{s(k)}\right)\frac{\alpha_k^2}{c_{k}^{2}}\Delta_{k}\nonumber\\
	&= (1-s(k))\mathbb{E}\left[\left\|\widetilde{\mathbf{x}}(k)\right\|^{2}\right]+\frac{\alpha_k^2}{c_{k}^{2}s(k)}\Delta_{k}+\frac{\alpha_k^2}{c_{k}^{2}}\Delta_{k}.
	\end{align}
	Choosing a $k$ such that $k\geq 2k_3+1$ for ease of analysis, from \eqref{eq:dis6.6}, we have,
	\begin{align}
	\label{eq:dis6.7}
	&\mathbb{E}\left[\left\|\widetilde{\mathbf{x}}(k+1)\right\|^{2}\right]\leq \prod_{l=k_3}^{k}(1-s(l))\mathbb{E}\left[\left\|\widetilde{\mathbf{x}}(k_3)\right\|^{2}\right]\nonumber\\
	&+\Delta_{k_{3}}\sum_{l=k_3}^{\lfloor\frac{k-1}{2}-1\rfloor}\prod_{m=l+1}^{k}(1-s(m))\left(\frac{\alpha_l^2}{c_{l}^{2}s(l)}+\frac{\alpha_l^2}{c_{l}^{2}}\right)\nonumber\\
	&+\Delta_{\lfloor\frac{k-1}{2}\rfloor}\sum_{l=\lfloor\frac{k-1}{2}\rfloor}^{k}\prod_{m=l+1}^{k}(1-s(m))\left(\frac{\alpha_l^2}{c_{l}^{2}s(l)}+\frac{\alpha_l^2}{c_{l}^{2}}\right)\nonumber\\
	&\le \exp\left(-\sum_{l=k_3}^{k}s(l)\right)\mathbb{E}\left[\left\|\widetilde{\mathbf{x}}(k_3)\right\|^{2}\right]\nonumber\\
	&+\Delta_{k_3}\prod_{m=\lfloor\frac{k-1}{2}\rfloor}^{k}(1-s(m))\sum_{l=k_3}^{\lfloor\frac{k-1}{2}\rfloor-1}\left(\frac{\alpha_l^2}{c_{l}^{2}s(l)}+\frac{\alpha_l^2}{c_{l}^{2}}\right)\nonumber\\
	&+\Delta_{\lfloor\frac{k-1}{2}\rfloor}\frac{\alpha_{\lfloor\frac{k-1}{2}\rfloor}^2}{c_{\lfloor\frac{k-1}{2}\rfloor}^{2}s^{2}(\lfloor\frac{k-1}{2}\rfloor)}\nonumber\\&\times\sum_{l=\lfloor\frac{k-1}{2}\rfloor}^{k}\left(\prod_{m=l}^{k}(1-s(m))-\prod_{m=l+1}^{k}(1-s(m))\right)\nonumber\\
	&+\Delta_{\lfloor\frac{k-1}{2}\rfloor}\frac{\alpha_{\lfloor\frac{k-1}{2}\rfloor}^2}{c_{\lfloor\frac{k-1}{2}\rfloor}^{2}s(\lfloor\frac{k-1}{2}\rfloor)}\nonumber\\&\times\sum_{l=\lfloor\frac{k-1}{2}\rfloor}^{k}\left(\prod_{m=l}^{k}(1-s(m))-\prod_{m=l+1}^{k}(1-s(m))\right)\nonumber\\
	&\le \underbrace{\exp\left(-\sum_{l=k_3}^{k}s(l)\right)\mathbb{E}\left[\left\|\widetilde{\mathbf{x}}(k_3)\right\|^{2}\right]}_{\text{$t_1$}}\nonumber\\
	&+\underbrace{\Delta_{k_3}\exp\left(-\sum_{m=\lfloor\frac{k-1}{2}\rfloor}^{k}s(m)\right)\sum_{l=k_3}^{\lfloor\frac{k-1}{2}\rfloor-1}\left(\frac{\alpha_l^2}{p_{\mathcal{L}}c_{l}^{2}\beta_{l}}+\frac{\alpha_l^2}{c_{l}^{2}}\right)}_{\text{$t_2$}}\nonumber\\
	&+\underbrace{\frac{4\Delta_{\lfloor\frac{k-1}{2}\rfloor}\alpha_{0}^{2}}{p_{\mathcal{L}}^{2}\beta_0^{2}c_0^2(k+1)^{2-2\tau-2\delta}}}_{\text{$t_3$}}+\underbrace{\frac{4\Delta_{\lfloor\frac{k-1}{2}\rfloor}\alpha_{0}^{2}}{p_{\mathcal{L}}\beta_0c_0^2(k+1)^{2-\tau-2\delta}}}_{\text{$t_4$}}.
	\end{align}
	In the above proof, the splitting in the interval $[k_3,k]$ was done at $\lfloor\frac{k-1}{2}\rfloor$ for ease of book keeping. The division can be done at an arbitrary point.
	It is to be noted that the sequence $\{s(k)\}$ is not summable and hence terms $t_1$ and $t_2$ decay faster than $(k+1)^{2-2\tau-2\delta}$. Also, note that term $t_4$ decays faster than $t_3$. Furthermore, $t_3$ can be written as
	{\small\begin{align*}
		&\frac{4\Delta_{\lfloor\frac{k-1}{2}\rfloor}\alpha_{0}^{2}}{p_{\mathcal{L}}^{2}\beta_0^{2}c_0^2(k+1)^{2-2\tau-2\delta}} = \underbrace{\frac{4\Delta_{1,\infty}\alpha_{0}^{2}}{p_{\mathcal{L}}^{2}\beta_0^{2}c_0^2(k+1)^{2-2\tau-2\delta}}}_{\text{$t_{31}$}}\nonumber\\&+\underbrace{\frac{c_{\lfloor\frac{k-1}{2}\rfloor}^{2}4\Delta_{2,\infty}\alpha_{0}^{2}}{p_{\mathcal{L}}^{2}\beta_0^{2}c_0^2(k+1)^{2-2\tau-2\delta}}}_{\text{$t_{32}$}},
		\end{align*}}
	from which we have that $t_{32}$ decays faster than $t_{31}$.
	For notational ease, henceforth we refer to $t_1+t_2+t_{32}+t_4=Q_{k}$, while keeping in mind that $Q_{k}$ decays faster than $(k+1)^{2-2\tau-2\delta}$.
\end{proof}
Hence, we have the disagreement given by,
\begin{align*}
\mathbb{E}\left[\left\|\widetilde{\mathbf{x}}(k+1)\right\|^{2}\right] \le =O\left(\frac{1}{k^{2-2\delta-2\tau}}\right).
\end{align*}
\end{proof}
\subsection{Proof of Theorem \ref{theorem-1}}
\label{sec:opt_gap}
\noindent In this subsection, we complete the proof of Theorem \ref{theorem-1} by characterizing the optimality gap of the network averaged optimizer estimate sequence and then combining it with the result obtained in Lemma \ref{lemma-disag-bound}.\\

\noindent Denote $\overline{\mathbf{x}}(k)=\frac{1}{N}\sum_{n=1}\mathbf{x}_{i}(k)$.
From \eqref{eq:dis1}, we have,\\
{\small\begin{align}
\label{eq:opt1}
&\overline{\mathbf{x}}(k+1) = \overline{\mathbf{x}}(k)\nonumber\\&-\frac{\alpha_{k}}{c_k}\left[\frac{c_k}{N}\sum_{i=1}^{N}\nabla f_{i}\left(\mathbf{x}_{i}(k)\right)+\underbrace{\frac{c_k^{2}}{N}\sum_{i=1}^{N}\mathbf{P}_{i}\left(\mathbf{x}_{i}(k)\right)}_{\text{$\overline{\mathbf{P}}\left(\mathbf{x}(k)\right)$}}+\underbrace{\frac{1}{N}\sum_{i=1}^{N}\mathbf{v}_{i}(k)}_{\text{$\overline{\mathbf{v}}(k)$}}\right]\nonumber\\
&\Rightarrow \overline{\mathbf{x}}(k+1) = \overline{\mathbf{x}}(k)\nonumber\\&-\frac{\alpha_k}{Nc_{k}}\left[\sum_{i=1}^{N}\nabla f_{i}\left(\mathbf{x}_{i}(k)\right)-\nabla f_{i}\left(\overline{\mathbf{x}}(k)\right)+\nabla f_{i}\left(\overline{\mathbf{x}}(k)\right)\right]\nonumber\\&-\frac{\alpha_{k}}{c_{k}}\left(\overline{\mathbf{v}}(k)+\overline{\mathbf{P}}\left(\mathbf{x}(k)\right)\right).
\end{align}}
Recall that $f(\cdot)=\sum_{i=1}^{N}f_{i}(\cdot)$.
Then, we have,
\begin{align}
\label{eq:opt2}
&\overline{\mathbf{x}}(k+1) = \overline{\mathbf{x}}(k)-\frac{\alpha_{k}}{N}\nabla f\left(\overline{\mathbf{x}}(k)\right)\nonumber\\&-\frac{\alpha_k}{N}\left[\sum_{i=1}^{N}\nabla f_{i}\left(\mathbf{x}_{i}(k)\right)-\nabla f_{i}\left(\overline{\mathbf{x}}(k)\right)\right]\nonumber\\&-\frac{\alpha_{k}}{c_{k}}\left(\overline{\mathbf{v}}(k)+\overline{\mathbf{P}}\left(\mathbf{x}(k)\right)\right)\nonumber\\
&\Rightarrow \overline{\mathbf{x}}(k+1) = \overline{\mathbf{x}}(k)-\frac{\alpha_k}{Nc_{k}}\left[c_{k}\nabla f\left(\overline{\mathbf{x}}(k)\right)+\mathbf{e}(k)\right],
\end{align}
where
\begin{align}
\label{eq:opt3}
&\mathbf{e}(k) = N\overline{\mathbf{v}}(k)\nonumber\\&+\underbrace{N\overline{P}\left(\mathbf{x}(k)\right)+c_{k}\sum_{i=1}^{N}\left(\nabla f_{i}\left(\mathbf{x}_{i}(k)\right)-\nabla f_{i}\left(\overline{\mathbf{x}}(k)\right)\right)}_{\text{$\boldsymbol{\epsilon}(k)$}}.
\end{align}
Note that, $c_{k}\left\|\nabla f_{i}\left(\mathbf{x}_{i}(k)\right)-\nabla f_{i}\left(\overline{\mathbf{x}}(k)\right)\right\| \leq c_{k}L\left\|\mathbf{x}_{i}(k)-\overline{\mathbf{x}}(k)\right\| = c_{k}L\left\|\widetilde{\mathbf{x}}_{i}(k)\right\|$. We also have that, $\left\|\overline{\mathbf{P}}\left(\mathbf{x}(k)\right)\right\| \le (L-\mu)\sqrt{d}c_{k}^{2}$. Thus, we can conclude that, $\forall k\geq k_3$
\begin{align}
\label{eq:opt4}
&\boldsymbol{\epsilon}(k) = c_{k}\sum_{i=1}^{N}\left(\nabla f_{i}\left(\mathbf{x}_{i}(k)\right)-\nabla f_{i}\left(\overline{\mathbf{x}}(k)\right)\right)+N\overline{P}\left(\mathbf{x}(k)\right)\nonumber\\
&\Rightarrow\left\|\boldsymbol{\epsilon}(k)\right\|^{2} \leq 2NL^{2}c_{k}^{2}\left\|\widetilde{\mathbf{x}}(k)\right\|^{2}+2(L-\mu)^{2}N^{2}dc_{k}^{4}\nonumber\\
&\Rightarrow\mathbb{E}\left[\left\|\boldsymbol{\epsilon}(k)\right\|^{2}\right] \leq \frac{8NL^{2}\Delta_{1,\infty}\alpha_{0}^{2}}{p_{\mathcal{L}}^{2}\beta_0^{2}(k+1)^{2-2\tau}}+\frac{2(L-\mu)^{2}N^{2}dc_{0}^{4}}{(k+1)^{4\delta}}\nonumber\\&+\frac{2NL^{2}Q_{k}c_{0}^{2}}{(k+1)^{2\delta}}.
\end{align}
With the above development in place, we rewrite \eqref{eq:opt2} as follows:
\begin{align}
\label{eq:opt5}
&\overline{\mathbf{x}}(k+1) = \overline{\mathbf{x}}(k)-\frac{\alpha_{k}}{N}\nabla f\left(\overline{\mathbf{x}}(k)\right)-\frac{\alpha_k}{Nc_{k}}\boldsymbol{\epsilon}(k)-\frac{\alpha_{k}}{c_{k}}\overline{\mathbf{v}}(k)\nonumber\\
&\Rightarrow \overline{\mathbf{x}}(k+1)-\mathbf{x}^{\ast} = \overline{\mathbf{x}}(k)-\mathbf{x}^{\ast}-\frac{\alpha_{k}}{N}\left[\nabla f\left(\overline{\mathbf{x}}(k)\right)-\underbrace{\nabla f\left(\mathbf{x}^{\ast}\right)}_{\text{$=0$}}\right]\nonumber\\&-\frac{\alpha_k}{Nc_{k}}\boldsymbol{\epsilon}(k)-\frac{\alpha_{k}}{c_{k}}\overline{\mathbf{v}}(k).
\end{align}
By Leibnitz rule, we have,
\begin{align}
\label{eq:opt6}
&\nabla f\left(\overline{\mathbf{x}}(k)\right)-\nabla f\left(\mathbf{x}^{\ast}\right) \nonumber\\&= \underbrace{\left[\int_{s=0}^{1}\nabla^{2}f\left(\mathbf{x}^{\ast}+s\left(\overline{\mathbf{x}}(k)-\mathbf{x}^{\ast}\right)\right)ds\right]}_{\text{$\overline{\mathbf{H}}_{k}$}}\left(\overline{\mathbf{x}}(k)-\mathbf{x}^{\ast}\right),
\end{align}
where it is to be noted that $NL\succcurlyeq\overline{\mathbf{H}}_{k}\succcurlyeq N\mu$.
Using \eqref{eq:opt6} in  \eqref{eq:opt5}, we have,
\begin{align}
\label{eq:opt7}
&\left(\overline{\mathbf{x}}(k+1)-\mathbf{x}^{\ast}\right) = \left[\mathbf{I}-\frac{\alpha_k}{N}\overline{\mathbf{H}}_{k}\right]\left(\overline{\mathbf{x}}(k)-\mathbf{x}^{\ast}\right)\nonumber\\&-\frac{\alpha_k}{Nc_{k}}\boldsymbol{\epsilon}(k)-\frac{\alpha_{k}}{c_{k}}\overline{\mathbf{v}}(k).
\end{align}
Denote by $\mathbf{m}(k)=\left[\mathbf{I}-\frac{\alpha_k}{N}\overline{\mathbf{H}}_{k}\right]\left(\overline{\mathbf{x}}(k)-\mathbf{x}^{\ast}\right)-\frac{\alpha_k}{Nc_{k}}\boldsymbol{\epsilon}(k)$ and note that $\mathbf{m}(k)$ is conditionally independent from $\overline{\mathbf{v}}(k)$ given the history $\mathcal{F}_{k}$. Then \eqref{eq:opt7} can be rewritten as:
\begin{align}
\label{eq:opt8}
&\left(\overline{\mathbf{x}}(k+1)-\mathbf{x}^{\ast}\right)=\mathbf{m}(k)-\frac{\alpha_{k}}{c_{k}}\overline{\mathbf{v}}(k)\nonumber\\
&\Rightarrow \left\|\overline{\mathbf{x}}(k+1)-\mathbf{x}^{\ast}\right\|^{2} \le \left\|\mathbf{m}(k)\right\|^{2}\nonumber\\&-2\frac{\alpha_{k}}{c_{k}}\mathbf{m}(k)^{\top}\overline{\mathbf{v}}(k)+\frac{\alpha_{k}^{2}}{c_{k}^{2}}\left\|\overline{\mathbf{v}}(k)\right\|^{2}.
\end{align}
Using the properties of conditional expectation and noting that $\mathbb{E}\left[\mathbf{v}(k)|\mathcal{F}_{k}\right]=\mathbf{0}$, we have,
{\small\begin{align}
\label{eq:opt9}
&\mathbb{E}\left[\left\|\overline{\mathbf{x}}(k+1)-\mathbf{x}^{\ast}\right\|^{2}|\mathcal{F}_{k}\right] \le \left\|\mathbf{m}(k)\right\|^{2}+\frac{\alpha_{k}^{2}}{c_k^2}\mathbb{E}\left[\left\|v(k)\right\|^{2}|\mathcal{F}_{k}\right]\nonumber\\
&\Rightarrow \mathbb{E}\left[\left\|\overline{\mathbf{x}}(k+1)-\mathbf{x}^{\ast}\right\|^{2}\right] \le \mathbb{E}\left[\left\|\mathbf{m}(k)\right\|^{2}\right]\nonumber\\&+
\frac{2\alpha_{k}^{2}\left(c_{f}q_{\infty}(N,d,\alpha_0,c_0)+N\sigma_1^{2}\right)}{c_k^2}.
\end{align}}
Using \eqref{eq:cool_ineq_1}, we have for $\mathbf{m}(k)$,
\begin{align}
\label{eq:opt10}
&\left\|\mathbf{m}(k)\right\|^{2} \le \left(1+\theta_{k}\right)\left\|\mathbf{I}-\frac{\alpha_k}{N}\overline{\mathbf{H}}_{k}\right\|^{2}\left\|\overline{\mathbf{x}}(k)-\mathbf{x}^{\ast}\right\|^{2}\nonumber\\&+\left(1+\frac{1}{\theta_k}\right)\frac{\alpha_k^2}{N^2c_{k}^{2}}\left\|\boldsymbol{\epsilon}(k)\right\|^{2}\nonumber\\
&\le \left(1+\theta_{k}\right)\left(1-\frac{\mu\alpha_{0}}{N(k+1)}\right)^{2}\left\|\overline{\mathbf{x}}(k)-\mathbf{x}^{\ast}\right\|^{2}\nonumber\\&+\left(1+\frac{1}{\theta_k}\right)\frac{\alpha_k^2}{N^2c_{k}^{2}}\left\|\boldsymbol{\epsilon}(k)\right\|^{2}.
\end{align}
On choosing $\theta_{k}=\frac{\mu\alpha_0}{N(k+1)}$, we have for all $k\geq k_3$, 
{\small\begin{align}
\label{eq:opt11}
&\mathbb{E}\left[\left\|\mathbf{m}(k)\right\|^{2}\right] \leq \left(1-\frac{\mu\alpha_0}{N(k+1)}\right)\mathbb{E}\left[\left\|\overline{\mathbf{x}}(k)-\mathbf{x}^{\ast}\right\|^{2}\right]\nonumber\\&+
\frac{16\Delta_{1,\infty}\alpha_{0}^{3}}{c_0^2p_{\mathcal{L}}^{2}\beta_0^{2}(k+1)^{3-2\tau-2\delta}}+\frac{4(L-\mu)^{2}Nd\alpha_{0}c_{0}^{2}}{(k+1)^{1+2\delta}}\nonumber\\&+\frac{4Q_{k}c_{0}^{2}}{(k+1)^{2\delta}}\nonumber\\
&\Rightarrow \mathbb{E}\left[\left\|\overline{\mathbf{x}}(k+1)-\mathbf{x}^{\ast}\right\|^{2}\right] \leq \left(1-\frac{\mu\alpha_0}{N(k+1)}\right)\nonumber\\&\times\mathbb{E}\left[\left\|\overline{\mathbf{x}}(k)-\mathbf{x}^{\ast}\right\|^{2}\right]\nonumber\\&+\frac{16\Delta_{1,\infty}\alpha_{0}^{3}}{c_0^2p_{\mathcal{L}}^{2}\beta_0^{2}(k+1)^{3-2\tau-2\delta}}+\frac{4(L-\mu)^{2}Nd\alpha_{0}c_{0}^{2}}{(k+1)^{1+2\delta}}\nonumber\\&+\frac{4Q_{k}c_{0}^{2}}{(k+1)^{2\delta}}+\frac{2\alpha_{0}^{2}\left(c_{f}q_{\infty}(N,d,\alpha_0,c_0)+N\sigma_1^{2}\right)}{c_0^2(k+1)^{2-2\delta}}.
\end{align}}
Proceeding as in \eqref{eq:dis6.7}, we have $\forall k \geq k_3$
{\small\begin{align}
\label{eq:opt11.5}
&\mathbb{E}\left[\left\|\overline{\mathbf{x}}(k+1)-\mathbf{x}^{\ast}\right\|^{2}\right]\nonumber\\
& \leq \underbrace{\exp\left(-\frac{\mu}{N}\sum_{l=k_5}^{k}\alpha_{l}\right)\mathbb{E}\left[\left\|\overline{\mathbf{x}}(k)-\mathbf{x}^{\ast}\right\|^{2}\right]}_{\text{$t_6$}}\nonumber\\
&+\underbrace{\exp\left(-\frac{\mu}{N}\sum_{m=\lfloor\frac{k-1}{2}\rfloor}^{k}\alpha_{m}\right)\sum_{l=k_5}^{\lfloor\frac{k-1}{2}\rfloor-1}\frac{16\Delta_{1,\infty}\alpha_{0}^{3}}{c_0^2p_{\mathcal{L}}^{2}\beta_0^{2}(l+1)^{3-2\tau-2\delta}}}_{\text{$t_7$}}\nonumber\\&+\underbrace{\exp\left(-\frac{\mu}{N}\sum_{m=\lfloor\frac{k-1}{2}\rfloor}^{k}\alpha_{m}\right)\sum_{l=k_5}^{\lfloor\frac{k-1}{2}\rfloor-1}\frac{4(L-\mu)^{2}Nd\alpha_{0}c_{0}^{2}}{(k+1)^{1+2\delta}}}_{\text{$t_8$}}\nonumber\\
&+\underbrace{\exp\left(-\frac{\mu}{N}\sum_{m=\lfloor\frac{k-1}{2}\rfloor}^{k}\alpha_{m}\right)\sum_{l=k_5}^{\lfloor\frac{k-1}{2}\rfloor-1}\frac{4Q_{k}c_{0}^{2}}{(k+1)^{2\delta}}}_{\text{$t_{9}$}}\nonumber\\
&+\underbrace{\exp\left(-\frac{\mu}{N}\sum_{m=\lfloor\frac{k-1}{2}\rfloor}^{k}\alpha_{m}\right)\sum_{l=k_5}^{\lfloor\frac{k-1}{2}-1}\frac{2\alpha_{0}^{2}c_{f}q_{\infty}(N,d,\alpha_0,c_0)}{c_0^2(k+1)^{2-2\delta}}}_{\text{$t_{10}$}}\nonumber\\
&+\underbrace{\exp\left(-\frac{\mu}{N}\sum_{m=\lfloor\frac{k-1}{2}\rfloor}^{k}\alpha_{m}\right)\sum_{l=k_5}^{\lfloor\frac{k-1}{2}-1}\frac{2\alpha_0^2N\sigma_1^{2}}{c_0^2(k+1)^{2-2\delta}}}_{\text{$t_{11}$}}\nonumber\\
&+\underbrace{\frac{64N\Delta_{1,\infty}\alpha_{0}^{2}}{\mu c_0^2p_{\mathcal{L}}^{2}\beta_0^{2}(k+1)^{2-2\tau-2\delta}}}_{\text{$t_{12}$}}\nonumber\\
&+\underbrace{\frac{4(L-\mu)^{2}N^{2}dc_{0}^{2}}{\mu(k+1)^{2\delta}}}_{\text{$t_{13}$}}+\underbrace{\frac{2Nc_{0}^{2}Q_{k}}{\mu\alpha_{0}(k+1)^{2\delta-1}}}_{\text{$t_{14}$}}\nonumber\\
&+\underbrace{\frac{4N\alpha_{0}\left(c_{f}q_{\infty}(N,d,\alpha_0,c_0)+N\sigma_1^{2}\right)}{c_0^2\mu(k+1)^{1-2\delta}}}_{\text{$t_{15}$}}.
\end{align}}
It is to be noted that the term $t_6$ decays exponentially. The terms $t_7$, $t_8$, $t_9$, $t_{10}$ and $t_{11}$ decay faster than its counterparts in the terms $t_{12}$, $t_{13}$, $t_{14}$ and $t_{15}$ respectively. We note that $Q_{l}$ also decays faster.
Hence, the rate of decay of $\mathbb{E}\left[\left\|\overline{\mathbf{x}}(k+1)-\mathbf{x}^{\ast}\right\|^{2}\right]$ is determined by the terms $t_{12}$, $t_{13}$ and $t_{15}$. Thus, we have that, $\mathbb{E}\left[\left\|\overline{\mathbf{x}}(k+1)-\mathbf{x}^{\ast}\right\|^{2}\right]=O\left(k^{-\delta_{1}}\right)$,
where $\delta_{1}=\min\left\{1-2\delta,2-2\tau-2\delta,2\delta\right\}$. For notational ease, we refer to $t_6+t_7+t_8+t_9+t_{10}+t_{11}+t_{14}= M_{k}$ from now on.
Finally, we note that,
\begin{align}
\label{eq:opt12}
&\left\|\mathbf{x}_{i}(k)-\mathbf{x}^{\ast}\right\| \le \left\|\overline{\mathbf{x}}(k)-\mathbf{x}^{\ast}\right\|+\left\|\underbrace{\mathbf{x}_{i}(k)-\overline{\mathbf{x}}(k)}_{\text{$\widetilde{\mathbf{x}}_{i}(k)$}}\right\|\nonumber\\
&\Rightarrow\left\|\mathbf{x}_{i}(k)-\mathbf{x}^{\ast}\right\|^{2} \leq 2\left\|\widetilde{\mathbf{x}}_{i}(k)\right\|^{2}+2\left\|\overline{\mathbf{x}}(k)-\mathbf{x}^{\ast}\right\|^{2}\nonumber\\
&\Rightarrow\mathbb{E}\left[\left\|\mathbf{x}_{i}(k)-\mathbf{x}^{\ast}\right\|^{2}\right] \le 2M_{k}+\frac{64N\Delta_{1,\infty}\alpha_{0}^{2}}{\mu c_0^2p_{\mathcal{L}}^{2}\beta_0^{2}(k+1)^{2-2\tau-2\delta}}\nonumber\\&\frac{4(L-\mu)^{2}N^{2}dc_{0}^{2}}{\mu(k+1)^{2\delta}}+2Q_{k}+
\frac{8\Delta_{1,\infty}\alpha_{0}^{2}}{p_{\mathcal{L}}^{2}\beta_0^{2}c_0^2(k+1)^{2-2\tau-2\delta}}\nonumber\\&+
\frac{4N\alpha_{0}\left(c_{f}q_{\infty}(N,d,\alpha_0,c_0)+N\sigma_1^{2}\right)}{c_0^2\mu(k+1)^{1-2\delta}}\nonumber\\
&\Rightarrow\mathbb{E}\left[\left\|\mathbf{x}_{i}(k)-\mathbf{x}^{\ast}\right\|^{2}\right] = O\left(\frac{1}{k^{\delta_{1}}}\right),~~\forall i,
\end{align}
where $\delta_{1}=\min\left\{1-2\delta,2-2\tau-2\delta,2\delta\right\}$. By, optimizing over $\tau$ and $\delta$, we obtain that for $\tau=1/2$ and $\delta=1/4$,
\begin{align*}
\mathbb{E}\left[\left\|\mathbf{x}_{i}(k)-\mathbf{x}^{\ast}\right\|^{2}\right] = O\left(\frac{1}{k^{\frac{1}{2}}}\right),~~\forall i.
\end{align*}

\section{Simulation Example}
\label{sec:sim}
\noindent We provide a simulation example pertaining to $\ell_2$-regularized logistic losses in random network characterized by link failures independent across iteration and links with probability $p_{\mathrm{fail}}$. To be specific, we consider $\ell_2$-regularized empirical risk minimization with logistic loss, where the regularization function is given by $\Psi_i(\mathbf{x})=
\frac{\kappa}{2}\|\mathbf{x}\|^2$, $i=1,...,N$, with $\kappa=0.3$. In our simulation setup, each node has access to $n_i = 10$ data points. The class labels and the classification vector given by 
$b_{ij}=\mathrm{sign} \left( (\mathbf{x}^\prime_1)^\top \mathbf{a}_{i,j}+x^\prime_0+\epsilon_{ij}\right)$ and $x^\prime=((\mathbf{x}_1^\prime)^\top, x_0^\prime)^\top$ respectively have $\epsilon_{ij}$s and the entries of $x^\prime$ drawn independently from standard normal distribution. The feature vectors $\mathbf{a}_{i,j}$, $j=1,...,n_i$, across different nodes $i=1,\cdots,N$ and across different entries are drawn independently from different distributions. To be specific, at node $i$, $\mathbf{a}_{i,j}$, $j=1,...,n_i$ is generated by adding a standard normal random variable and an uniform random variable with support $[0,\,5\,i]$.

\noindent We set $\beta_k=\frac{1}{\theta\,(k+1)^{1/2}}$, $\alpha_k=\frac{1}{k+1}$, $c_{k}= \frac{1}{(k+1)^{1/4}}$, $k=0,1,...$, where $\theta=7$ is the maximum degree across nodes. The optimizer estimate at each node is initialized as $\mathbf{x}_i(0)=0$, $\forall~i=1,...,N$.

\noindent We consider a connected network $\mathcal{G}$ with $N=10$ nodes and $23$ links, generated as an instance of a random geometric graph. The random network model assumes link failures independent across iterations and links with probability~$p_{\mathrm{fail}}$, where $p_{\mathrm{fail}} \in \{0;\,0.5;\,0.7\}$. The case $p_{\mathrm{fail}}=0$ corresponds to the case where none of the links fail. We also include a comparison with the centralized zeroth order KWSA based optimization method:
{\small\begin{equation}
\label{eqn-centralized-SGD}
\mathbf{y}(k+1) = \mathbf{y}(k) -
\frac{1}{N(k+1)}
\sum_{i=1}^N \nabla g_{i}\left(\,\mathbf{y}(k);\,\mathbf{a}_i(k),b_i(k)\,\right),
\end{equation}}
\noindent where $(\mathbf{a}_i(k),b_i(k))$ is drawn uniformly from the set $(\mathbf{a}_{i,j},b_{i,j})$, $j=1,...,n_i$. Algorithm~\eqref{eqn-centralized-SGD} shows how~\eqref{eq:update_rule_node} would be implemented if there existed a fusion node with access to all nodes' data.
Hence, the comparison with~\eqref{eqn-centralized-SGD} allows us to study the degradation of \eqref{eq:update_rule_node} due to lack of global model information. The step size for \eqref{eqn-centralized-SGD} is set to $1/N(k+1)$. As an error metric, we use the mean square error (MSE) estimate averaged across nodes: $\frac{1}{N} \sum_{i=1}^N {\|\mathbf{x}_i(k)-{\mathbf{x}}^\star\|^2}. $

\noindent Figure~\ref{Figure_1} plots the estimated MSE, averaged across 100 algorithm runs, versus iteration number~$k$ for $p_{\mathrm{fail}} \in \{0;\,0.5;\,0.7\}$ in $\mathrm{log}_{10}$-$\mathrm{log}_{10}$ scale. The slope of the plot curve corresponds to the sublinear rate of the method; e.g., the $-1/2$ slope corresponds to a $1/k^{0.5}$ rate.
It is to be noted that for all values of~$p_{\mathrm{fail}}$, the  algorithm \eqref{eq:update_rule_node} achieves on this example (at least) the $1/k^{0.5}$ rate, thus corroborating our theory.
The increase of the link failure probability only increases the constant in the MSE but does not affect the rate but the curves are only vertically shifted. Interestingly, the loss due to
the increase of $p_{\mathrm{fail}}$ is small; e.g., the curves that correspond to $p_{\mathrm{fail}}=0.5$ and $p_{\mathrm{fail}}=0$ (no link failures) practically match.
Figure~\ref{Figure_1} also shows the performance of the centralized method~\eqref{eqn-centralized-SGD}.
We can see that, the distributed method \eqref{eq:update_rule_node} is very close in performance to the centralized method.
\begin{figure}[thpb]
	\centering
	\includegraphics[height=2.8 in,width=3.4 in]{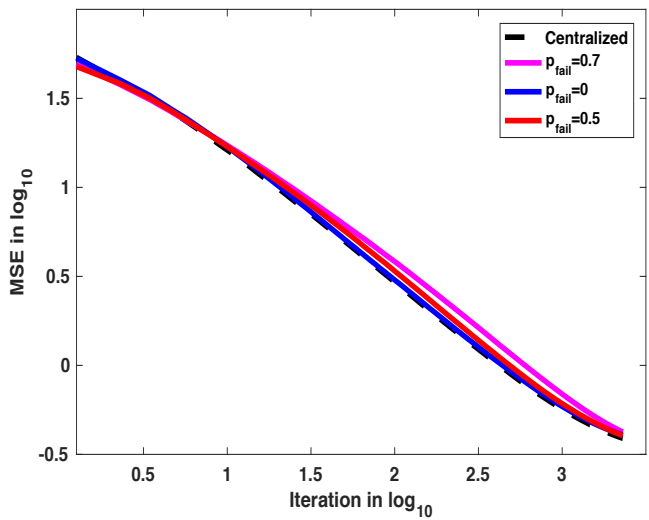}
	\caption{Estimated MSE versus iteration number~$k$ for algorithm \eqref{eq:update_rule_node}
		with link failure probability $p_{\mathrm{fail}}=0$ (blue, solid line);
		$0.5$ (red, solid line); and $0.7$ (pink, solid line).
		The Figure also shows the performance of
		the centralized stochastic gradient method in~\eqref{eqn-centralized-SGD}
		(black, dashed line).}
	\label{Figure_1}
	\vspace{-5mm}
\end{figure}
\section{Conclusion}
\label{sec:conc}
\noindent We have considered a distributed stochastic zeroth order optimization method for smooth strongly convex optimization, where we have employed a Kiefer Wolfowitz stochastic approximation type algorithm. Through the analysis of the considered method, we have established the $O(1/k^{1/2})$ MSE convergence rate for the assumed optimization setting when the underlying network is
randomly varying. In particular, we have also quantified the mean square error of the generated optimizer estimate sequence in terms of the algorithm parameters. Future work includes extending the current approach to general class of convex and non-convex functions.

\bibliographystyle{IEEEtran}
\bibliography{IEEEabrv,CentralBib,dsprt,glrt}

\end{document}